\documentclass[reqno]{amsart}
\usepackage{amssymb}
\usepackage{amsfonts}
\usepackage{amssymb}
\usepackage{xcolor}
\newtheorem{theorem}{Theorem}[section]
\newtheorem{proposition}[theorem]{Proposition}
\newtheorem{lemma}[theorem]{Lemma}

\theoremstyle{definition}
\newtheorem{definition}[theorem]{Definition}
\newtheorem{example}[theorem]{Example}

\newtheorem{remark}[theorem]{Remark}

\allowdisplaybreaks
\numberwithin{equation}{section}
\begin{document}
	\title[Explicit hypergeometric modularity]
	{Explicit hypergeometric modularity of certain weight two and four Hecke eigenforms}
	\author{Sipra Maity}
	\address{Department of Mathematics, Indian Institute of Technology Guwahati, North Guwahati, Guwahati-781039, Assam, INDIA}
	\curraddr{}
	\email{m.sipra@iitg.ac.in}
	\author{Rupam Barman}
	\address{Department of Mathematics, Indian Institute of Technology Guwahati, North Guwahati, Guwahati-781039, Assam, INDIA}
	\curraddr{}
	\email{{\tiny rupam@iitg.ac.in}}
	\thanks{}
	\date{\today}
	\begin{abstract} 
	Recently, Allen \emph{et al.} developed the Explicit Hypergeometric Modularity Method (EHMM) that establishes the modularity of a large class of hypergeometric Galois representations in dimensions two and three. Motivated by this framework, we construct two explicit families of eta-quotients, which we call the $\mathbb{K}_4$ and $\mathbb{K}_5$ functions, from the hypergeometric background. These $\mathbb{K}_4$ and $\mathbb{K}_5$ functions are constructed using the theory of weight $1/2$ Jacobi theta functions and their cubic analogues, respectively. Using these constructions, we then express the Fourier coefficients of certain Hecke eigenforms of weight two and four in terms of finite field period functions. As an application, we obtain new identities expressing the Fourier coefficients of weight two Hecke eigenforms in terms of special values of finite field analogues of the Appell series $F_1$ and $F_2$. These results provide new explicit connections between hypergeometric character sums, modular forms, and multivariable hypergeometric functions over finite fields, thereby extending the scope of known instances of hypergeometric modularity.
	\end{abstract}
	\maketitle
	\section{Introduction and statement of results}
		For a complex number $a$, the rising factorial $(a)_k$ is defined as $(a)_0:=1$ and $(a)_k:=a(a+1)\cdots (a+k-1)$ for $k\geq 1$. For a non-negative integer $n$, and $a_i, b_i\in\mathbb{C}$ with $b_i\notin\{\ldots, -3,-2,-1, 0\}$, the classical hypergeometric series ${_{n+1}}F_{n}$ is given by
	\begin{align*}
		{_{n+1}}F_{n}\left(\begin{array}{cccc}
			a_0, & a_1, & \ldots, & a_{n} \\
			& b_1, & \ldots, & b_n
		\end{array}\mid t
		\right):=\sum_{k=0}^{\infty}\frac{(a_0)_k\cdots (a_{n})_k}{(b_1)_k\cdots(b_n)_k}\cdot\frac{t^k}{k!},
	\end{align*}
	which converges absolutely for $|t|<1$. Classical hypergeometric series arise naturally as solutions to differential equations, and they appear in many areas of mathematics (see for example \cite{Andrews}).
	\par Let $p$ be an odd prime and $\mathbb{F}_q$ the finite field with $q = p^r$ elements, $r \geq 1$. Denote by $\widehat{\mathbb{F}_q^\times}$ the group of multiplicative characters of $\mathbb{F}_q^\times$, with trivial and quadratic characters $\varepsilon$ and $\varphi$, respectively. The arithmetic properties of Gauss and Jacobi sums have a very long history, and number theorists have used these sums to introduce finite field analogues of classical hypergeometric series. It seems that hypergeometric functions over a finite field first appeared in Koblitz's work \cite{kob}. There are other definitions of finite field hypergeometric functions. Here we make use of these functions, as developed by Greene \cite{greene} and Fuselier \emph{et al.} \cite{FL}.
	\par 
	For a multiplicative character $\chi$, we extend $\chi$ to $\mathbb{F}_q$ by setting $\chi(0):= 0$. For multiplicative characters $R_1$ and $R_2$, the binomial coefficient ${R_1 \choose R_2}$ is defined as
	\begin{align}\label{bino-def}
		{R_1 \choose R_2}:=\frac{R_2(-1)}{q}J(R_1,\overline{R_2})=\frac{R_2(-1)}{q}\sum_{x \in \mathbb{F}_q}R_1(x)\overline{R_2}(1-x),
	\end{align}
	where $J(\cdot, \cdot)$ denotes the usual Jacobi sum and $\overline{R_2}$ is the  character inverse of $R_2$. For multiplicative characters $R_1, R_2,\ldots, R_{n+1}$, $ Q_2,\ldots, Q_{n+1}$ and $z\in\mathbb{F}_q$,  Greene defined the \emph{Gaussian hypergeometric series ${_{n+1}}F_n$ over $\mathbb{F}_q$} by
	\begin{align*}
		_{n+1}F_{n}\left(\begin{array}{cccc}
			R_1 & R_2&\ldots  & R_{n+1} \\
			& Q_2&\ldots& Q_{n+1}
		\end{array}|z
		\right)
		:= \frac{q}{q-1}\sum_{\chi \in \widehat{{\mathbb{F}_q^\times}}} {R_{1}\chi\choose\chi} {R_{2}\chi\choose Q_{2}\chi} \cdots {R_{n+1}\chi\choose Q_{n+1}\chi}\chi(z).
	\end{align*}  In a more recent study \cite{FL}, Fuselier \emph{et al.} developed a version of hypergeometric functions over finite fields, paralleling the approach taken in the classical setting by considering period functions of algebraic varieties over finite fields. Following \cite{FL}, for $z\in\mathbb{F}_q^\times$, define the period functions by
	$$
	_{1}\mathbb{P}_{0}[R_{1}|z] := \overline{R_1}(1-z),
	$$
	and for $n \geq 1$,
	\begin{align*}
		&_{n+1}\mathbb{P}_{n}\left[\begin{array}{cccc}
			R_1 & R_2&\cdots  & R_{n+1} \\
			& Q_2&\cdots& Q_{n+1}
		\end{array}|z
		\right]\notag \\
		&:= \frac{q^{n+1}}{q-1} \left(\prod_{i=2}^{n+1}R_{i}Q_{i}(-1) \right) \sum_{\chi \in \widehat{{\mathbb{F}_q^\times}}} {R_{1}\chi\choose\chi} {R_{2}\chi\choose Q_{2}\chi} \cdots {R_{n+1}\chi\choose Q_{n+1}\chi}\chi(z). 
	\end{align*} 
	With the normalization \eqref{bino-def}, the above definition of ${_{n+1}}\mathbb{P}_n$ series differs from its original definition given in \cite{FL} by a factor of $(-q)^{n+1}$. The hypergeometric function defined by Greene and the period function defined by Fuselier \emph{et al.} are closely connected. For $z\in\mathbb{F}_q^\times$, we have
	\begin{align}\label{rel-gre-fu}
			&{_{n+1}}\mathbb{P}_{n}\left[\begin{array}{cccc}
				R_1 & R_2&\cdots  & R_{n+1} \\
				& Q_2&\cdots& Q_{n+1}
			\end{array}|z
			\right]\notag \\
			&=q^{n}\left(\prod_{i=2}^{n+1} R_{i} Q_{i}(-1) \right)\,{_{n+1}}F_{n}\left(\begin{array}{cccc}
				R_1 & R_2&\cdots  & R_{n+1} \\
				& Q_2&\cdots& Q_{n+1}
			\end{array}|z
			\right).
	\end{align}
    \par Throughout, for a positive integer $n$, we consider the multisets $\alpha = \{r_{1}, \ldots, r_{n}\}$ and $\beta = \{q_1=1, q_2, \ldots, q_{n}\}$ of rational numbers. 
    The collection $\mathrm{HD}:= \{\alpha, \beta\}$ is called a \emph{hypergeometric datum}. The integer $n$ is called the \emph{length} of $\mathrm{HD}$. The pair $(\alpha, \beta)$ is said to be \emph{primitive} if $r_{i}-q_{j} \notin \mathbb{Z}$ for all $1 \leq i,j \leq n$. Further, the least positive common denominator of $\mathrm{HD}$ is defined as $M:=\mathrm{lcd}(\alpha\cup\beta)$.
	To relate this data to finite fields, let $\mathbb{Z}[\zeta_M]$ be the cyclotomic ring, where $\zeta_M$ is a primitive $M$-th root of unity. For each nonzero prime ideal $\mathfrak{p}$ coprime to $M$ and integer $i$ we associate a multiplicative character to the finite residue field $\kappa_\mathfrak{p}:=\mathbb{Z}[\zeta_M]/\mathfrak{p}$ of size $q$ using the $M$-th residue symbol by 
	\begin{align*}
		\iota_\mathfrak{p}\left(\frac iM\right)(x) := \left( \frac {x}{ \mathfrak{p}}\right)_M^i\equiv x^{(q-1)\frac{i}M} \pmod {\mathfrak{p}}, \quad  \forall x\in \mathbb{Z}[\zeta_M].
	\end{align*} 
   For instance, when $\mathfrak{p}$ is coprime to $2$, the character $\iota_{\mathfrak{p}}(1/2)$ coincides with the quadratic character on the residue field. Likewise, $\iota_{\mathfrak{p}}(1)$ coincides with the trivial character. For simplicity, we write $R_i$ for $\iota_\mathfrak{p}(r_i)$, $\overline{R_i}$ for $\iota_\mathfrak{p}(-r_i)$, and define $Q_i$ and $\overline{Q_i}$ analogously. For a fixed $z\in\mathbb{F}_q^{\times}$ and $n\geq 1$, Fuselier \emph{et al.} \cite{FL} defined
   \begin{align}\label{P-definition2}
   \mathbb{P}(\text{HD};z;\mathfrak{p}) := {_{n}\mathbb{P}_{n-1}}\left[\begin{array}{cccc}
 	R_1 & R_2&\ldots  & R_{n} \\
 	& Q_2&\ldots& Q_{n}
   \end{array}|z
   \right].
   \end{align}
   When the length $n$ pair $(\alpha,\beta)$ is primitive, then the hypergeometric function $H_q$ is defined by \cite[(2.4)]{allen_EHMM1}
   \begin{align}\label{H_q-definition}
    H_q\left(\mathrm{HD};z;\mathfrak{p}\right):=\frac{\mathbb{P}(\text{HD};z;\mathfrak{p})}{\prod_{i=2}^{n}J_{\mathfrak{p}}(r_i,q_i-r_i)},
   \end{align}
   where $J_{\mathfrak{p}}(r_i,q_i-r_i):=J(\iota_{\mathfrak{p}}(r_i), \iota_{\mathfrak{p}}(q_i-r_i))$.
   Let $\mathbb{Q}_p$ denote the field of $p$-adic number. Let $\overline{\mathbb{Q}_p}$ be the algebraic closure of $\mathbb{Q}_p$ and $\mathbb{C}_p$ be the completion of $\overline{\mathbb{Q}_p}$. We will fix an embedding of $\mathbb{Q}(\zeta_M)$ to $\overline{\mathbb{C}_p}$ by taking $\iota_\mathfrak{p}\left(\frac{1}{q-1}\right):=\overline{\omega_q}$ to be the inverse of the Teichm\"{u}ller character $\omega_q$. Under this embedding, we write $\mathbb{P}(\mathrm{HD};z;\mathfrak{p})$ as $\mathbb{P}(\mathrm{HD};z;\overline{\omega_q})$. Likewise, we use $J_{\overline{\omega_q}}(r_i,q_i)$ for the embedding of $J_{\mathfrak{p}}(r_i,q_i)$.
  \par 
  Some of the principal motivations for studying Gaussian hypergeometric series arise from their deep connections with Fourier coefficients and eigenvalues of modular forms, as well as with point-counting on certain classes of algebraic varieties (see, for example, \cite{ahlgren_modularity, ahlgren, allen_modularity, BK, BK1, evans-mod, frechette, fuselier, Fuselier-McCarthy, koike, lennon, lennon2, mccarthy4, mc-papanikolas, mortenson, ono, salerno, vega}). A primary method for relating Fourier coefficients of modular forms and traces of Hecke operators to hypergeometric functions proceeds via the Eichler-Selberg trace formula, together with identities for $H_q$-functions and their connections to elliptic curves. However, this approach is effective only in a limited range of cases. For a broader overview of known results on hypergeometric modularity, we refer the reader to \cite{mccarthy_conjecture}.
  \par In the absence of a general framework for establishing hypergeometric modularity, Allen \emph{et al.} \cite{allen_EHMM2,allen_EHMM1} recently introduced the Explicit Hypergeometric Modularity Method (EHMM), which provides a systematic link between hypergeometric character sums and the Fourier coefficients of Hecke eigenforms. A key feature of the EHMM is the existence of an explicit formula for a Hecke eigenform $f_{\text{HD}}^{\sharp}$ in many instances when the datum $\text{HD}$ has length three or four with parameter $z= 1$. Moreover, the explicit nature of these hypergeometric realizations has led to a wealth of new identities and transformations for special values of $L$-functions associated to holomorphic modular forms \cite{allen_EHMM2,allen_EHMM1,rosen1,rosen2}. 
   \par
   This article concerns several new hypergeometric modularity results. To be specific, we construct two explicit families of eta-quotients, which we call the $\mathbb{K}_4$ and $\mathbb{K}_5$ functions, from the hypergeometric background. These $\mathbb{K}_4$ and $\mathbb{K}_5$ functions are constructed using the theory of weight $1/2$ Jacobi theta functions and their cubic analogues, respectively. See Section 3 for the constructions of  $\mathbb{K}_4$ and $\mathbb{K}_5$ functions. Using these constructions, we then prove identities between the $\mathbb{P}\left(\text{HD};z;\mathfrak{p}\right)$ function and the coefficients $a_{p}(f_{\text{HD}}^{\sharp})$ of a Hecke eigenform $f_{\mathrm{HD}}^{\sharp}$ at primes $p \equiv 1 \pmod{M}$. 
   Let us consider a particular form of hypergeometric datum HD of length four $$	\text{HD}_{\mathbb{K}_{4}\left(r\right)}:=\{\{1/2,1/2,1/2,r\},\{1,1,1,r+1/2\}\}.$$
   Our first result of this article expresses  the coefficients of weight four Hecke eigenforms in terms of $\mathbb{P}(\text{HD}_{\mathbb{K}_{4}(r)};z;\mathfrak{p})$ function.
   \begin{theorem}\label{main-thm-K4}
    Let $r \in \{1/2,1/3,1/4,1/6,1/8,1/12,1/24\}$. Let $f_{\mathrm{HD}_{\mathbb{K}_{4}\left(r\right)}}^{\sharp}$ be an explicit Hecke eigenform of weight four constructed from the EHMM. Then, for any prime ideal $\mathfrak{p}$ in $\mathbb{Z}[\zeta_{M}]$ lying above each fixed prime $p \equiv 1 \pmod{M:=\mathrm{lcd}{(1/2,r)}}$, we have
	$$\mathbb{P}\left(	\mathrm{HD}_{\mathbb{K}_{4}\left(r\right)};1;\mathfrak{p}\right)+\iota_{\mathfrak{p}}(1/2)(-1)p =-\psi(r)(\mathfrak{p})  \cdot a_{p}\left(f_{	\mathrm{HD}_{\mathbb{K}_{4}\left(r\right)}}^{\sharp}\right).$$
	The corresponding values of  $\psi{(r)}(\mathfrak{p})$ and $f_{	\mathrm{HD}_{\mathbb{K}_{4}\left(r\right)}}^{\sharp}$ are listed in Table~\ref{main-thm-table1}. Note the $L$-functions and modular forms database (LMFDB)  labels \cite{LMFDB} are used for all modular forms.
	\begin{table}[ht]
			\begin{center}
				\begin{tabular}{|c|c|c|c|}
					\hline
					\textrm{value of $r$} & $M$ & $\psi({r})(\mathfrak{p})$ & $f_{	\mathrm{HD}_{\mathbb{K}_{4}\left(r\right)}}^{\sharp}$  \\ \hline
					$1/2$ & $2$ & $\iota_{\mathfrak{p}}(1/2)(-1)$ & $f_{8.4.a.a}$\\
					\hline
					$1/3$ & $6$ & $1$ & $f_{9.4.a.a}$\\ 
					\hline
					$1/4$ & $4$ & $1$ & $f_{32.4.a.a},$~ $f_{32.4.a.c}$\\
					\hline
					$1/6$ & $6$ & $\iota_{\mathfrak{p}}(1/6)(-1)$ & $f_{72.4.a.a},$ ~$f_{72.4.a.d}$\\
					\hline
					$1/8$ & $8$ & $\iota_{\mathfrak{p}}(1/8)(-64)$ & $f_{128.4.b.e}$\\
					\hline
					$1/12$ & $12$ & $\iota_{\mathfrak{p}}(1/12)(-64)$ & $f_{288.4.a.l}$, $f_{288.4.a.m}$\\
					\hline
					$1/24$ & $24$ & $\iota_{\mathfrak{p}}(1/24)(-64)$ & $f_{1152.4.d.q}$\\
					\hline
				\end{tabular}
			\end{center}
			\caption{The $f_{\mathrm{HD}_{\mathbb{K}_{4}\left(r\right)}}^{\sharp}$ functions}\label{main-thm-table1}
	\end{table}
	\renewcommand{\arraystretch}{1}
	\end{theorem}
	\begin{remark}
	The cases $r=1/2$ and $r=1/4$ of Theorem~\ref{main-thm-K4} were obtained by Allen \emph{et al.} \cite{allen_EHMM1}.
	\end{remark}
	Next we consider another particular form of hypergeometric datum HD of length two $$	\text{HD}_{\mathbb{K}_{5}\left(r\right)}:=\{\{1/3,r\},\{1,1\}\}.$$
	We express the coefficients of weight two Hecke eigenforms in terms of $\mathbb{P}\left(\text{HD}_{\mathbb{K}_{5}(r)};z;\mathfrak{p}\right)$ functions in the following theorem.
	\begin{theorem}\label{main-thm-K5}
	Let $r\in\{1/2,1/3,1/4,1/6,1/12\}$. Let $f_{\mathrm{HD}_{\mathbb{K}_{5}\left(r\right)}}^{\sharp}$ be a Hecke eigenform of weight two  constructed from the EHMM. Then, for any prime ideal $\mathfrak{p}$ in $\mathbb{Z}[\zeta_{M^\prime}]$ lying above each fixed prime  $p\equiv 1 \pmod {M^\prime}$, we have
	\begin{align}\label{main-thm-K5-eq1}
	-a_{p}\left(f_{	\mathrm{HD}_{\mathbb{K}_{5}\left(r\right)}}^{\sharp}\right)=\iota_{\mathfrak{p}}(r)(27)\mathbb{P}\left(\mathrm{HD}_{\mathbb{K}_{5}\left(r\right)};1;\mathfrak{p}\right)+\iota_{\mathfrak{p}}(1-r)(27)\mathbb{P}\left(\overline{\mathrm{HD}_{\mathbb{K}_{5}\left(r\right)}};1;\mathfrak{p}\right).
	\end{align}
	Here, $M^\prime$ denotes the least common denominator of $1/3$ and $r$, and $\overline{\mathrm{HD}_{\mathbb{K}_{5}\left(r\right)}}:=\{\{2/3,1-r\};\{1,1\}\}$, while the corresponding values of $f_{	\mathrm{HD}_{\mathbb{K}_{5}\left(r\right)}}^{\sharp}$  are listed in Table~\ref{main-thm-table2}.
	\begin{table}[ht]
				\begin{center}
					\begin{tabular}{|c|c|c|}
						\hline
						\textrm{values of $r$} & $M^\prime$  & $f_{	\mathrm{HD}_{\mathbb{K}_{5}\left(r\right)}}^{\sharp}$  \\ \hline
						$1/2$, $1/6$ & $6$ &  $f_{36.2.a.a}$\\
						\hline
						$1/3$ & $3$ &  $f_{27.2.a.a}$\\
						\hline
						$1/4$ & $12$ &  $f_{144.2.c.a}$\\ 
						\hline
						$1/12$ & $12$ &  $f_{432.2.c.a}$\\
						\hline
					\end{tabular}
				\end{center}
				\caption{The $f_{	\mathrm{HD}_{\mathbb{K}_{5}\left(r\right)}}^{\sharp}$ functions}\label{main-thm-table2}
	\end{table}
	\renewcommand{\arraystretch}{1}
	\end{theorem}
	\subsection{Appell series over finite fields and modular forms}
	Products of $_2F_1$ classical hypergeometric series give rise to double series, among which the Appell functions $F_1$, $F_2$, $F_3$, and $F_4$ are the most prominent (see, for example, \cite{Andrews}). Finite field analogues of these functions have been studied by several authors (see, for example, \cite{He-et-al, TB-2, TSB}). Let $R_1, \ldots, R_5$ be multiplicative characters on $\mathbb{F}_q$,  and let $x,y\in \mathbb{F}_q$. Motivated by the integral representations of $F_1$ and $F_2$, He \emph{et al.} \cite{He-et-al} introduced finite field analogues of $F_1$ and $F_2$ as follows.
	\begin{align*}
		&F_1^q(R_1; R_2, R_3; R_4; x, y)\\
		&=\varepsilon(xy)R_1R_4(-1)\sum_{u\in \mathbb{F}_q}R_1(u)\overline{R_1}R_4(1-u)\overline{R_2}(1-ux)\overline{R_3}(1-uy),\\
		&F_2^q(R_1; R_2, R_3; R_4, R_5; x, y)\\
		&=\varepsilon(xy)R_2R_3R_4R_5(-1)\sum_{u, v\in \mathbb{F}_q}R_2(u)R_3(v)\overline{R_2}R_4(1-u)\overline{R_3}R_5(1-v)\overline{R_1}(1-ux-vy).
	\end{align*}
	Similar to \eqref{P-definition2}, for any rational numbers $r_1,r_2,r_3,r_4$, and $r_5$, we define
	\begin{align*}
		&F_1^q(r_1; r_2, r_3; r_4; x, y;\mathfrak{p}):=F_1^q(R_1; R_2, R_3; R_4; x, y),\\
		&F_2^q(r_1; r_2, r_3; r_4, r_5; x, y;\mathfrak{p}):=F_2^q(R_1; R_2, R_3; R_4, R_5; x, y),
	\end{align*}
	for each nonzero prime ideal $\mathfrak{p}$ in $\mathbb{Z}[\zeta_M]$ coprime to $M$, where $M$ is the least positive common denominator of $r_i$, and $R_i:=\iota_{\mathfrak{p}}(r_i)$ for $1\leq i\leq 5$.
	It is known that finite field analogues of the Appell series $F_4$ are related to modular forms (see, for example, \cite{MT}). To the best of our knowledge, no such results exist for the finite field analogues of the other Appell series. As an application of Theorem~\ref{main-thm-K5}, we establish explicit relations between the Fourier coefficients of weight two Hecke eigenforms and special values of the Appell series $F_1^p$ and $F_2^p$.
	\begin{theorem}\label{rel-appell-modu-1}
	Let $r\in\{1/2,1/3,1/4,1/6,1/12\}$. Let $f_{\mathrm{HD}_{\mathbb{K}_{5}\left(r\right)}}^{\sharp}$ be a Hecke eigenform of weight 2  constructed from the EHMM. Then, for any prime ideal $\mathfrak{p}$ in $\mathbb{Z}[\zeta_{M^\prime}]$ lying above each fixed prime  $p\equiv 1 \pmod {M^\prime}$, we have
	\begin{align*}
\iota_{\mathfrak{p}}(r)(-27)F_1^p\left(r;1/6,1/6;1;1,1;\mathfrak{p}\right)+\iota_{\mathfrak{p}}(1-r)(-27)F_1^p\left(1-r;1/3,1/3;1;1,1;\mathfrak{p}\right)\\	=-a_{p}\left(f_{	\mathrm{HD}_{\mathbb{K}_{5}\left(r\right)}}^{\sharp}\right).
	\end{align*}
	Here, $M^\prime =6$ for $r\in\{1/2,1/3,1/6\}$ and $M^\prime=12$ for $r\in\{1/4,1/12\}$. The corresponding values of $f_{\mathrm{HD}_{\mathbb{K}_{5}\left(r\right)}}^{\sharp}$  are listed in Table~\ref{main-thm-table2}.
	\end{theorem}
	\begin{theorem}\label{rel-appell-modu-2}
	Let $c^\prime\in\mathbb{Q}$ and $r\in\{1/2,1/3,1/4,1/6,1/12\}$. Let $f_{\mathrm{HD}_{\mathbb{K}_{5}\left(r\right)}}^{\sharp}$ be a Hecke eigenform of weight 2  constructed from the EHMM.  Then, for any prime ideal $\mathfrak{p}$ in $\mathbb{Z}[\zeta_{M^\prime}]$ lying above each fixed prime  $p\equiv 1 \pmod {M^\prime:=\mathrm{lcd}\left(1/3,r, c^\prime\right)}$, we have
	\begin{align*}
	a_{p}\left(f_{	\mathrm{HD}_{\mathbb{K}_{5}\left(r\right)}}^{\sharp}\right)&=\iota_{\mathfrak{p}}(r)(-27)\iota_{\mathfrak{p}}(c^\prime)(-1)F_2^p\left(1/3;r,1;1,c^\prime;1,1;\mathfrak{p}\right)+X\\&+\iota_{\mathfrak{p}}(1-r)(-27)\iota_{\mathfrak{p}}(c^\prime)(-1)F_2^p\left(2/3;1-r,1;1,c^\prime;1,1;\mathfrak{p}\right),
	\end{align*}
	where 
	\begin{align*}
	X:=-\iota_{\mathfrak{p}}(r)(-27)J_{{\mathfrak{p}}}(c^\prime,1/3-c^\prime)J_{\mathfrak{p}}(r+c^\prime-1/3,1/3-c^\prime)\\-\iota_{\mathfrak{p}}(1-r)(-27)J_{{\mathfrak{p}}}(c^\prime,2/3-c^\prime)J_{\mathfrak{p}}(1/3-r+c^\prime,2/3-c^\prime),
	\end{align*}
	and the corresponding values of $f_{	\mathrm{HD}_{\mathbb{K}_{5}\left(r\right)}}^{\sharp}$  are listed in Table~\ref{main-thm-table2}.
	\end{theorem}
    \section{Preliminaries}
	We begin by recalling  an alternate perspective on classical hypergeometric series in the form of the integral representation of Euler.
	\subsection{Euler's integral representation}
	Let $n$ be a non-negative integer and  $a_i, b_i\in\mathbb{C}$ such that $b_i\notin\{\ldots, -3,-2,-1, 0\}$. Then, there is an inductive formula to construct hypergeometric series, see \cite[(2.2.2)]{Andrews}. Namely, when $\text{Re}(b_{n})>\text{Re}(a_{n+1})>0$,
	\begin{align}\label{euler}
			{_{n+1}}F_{n}\left(\begin{array}{cccc}
				a_1 & a_2&\ldots  & a_{n+1} \\
				& b_1&\ldots& b_{n}
	\end{array}|z
	\right)= \frac{\Gamma(b_{n})}{\Gamma(a_{n+1})\Gamma(b_{n}-a_{n+1})}\notag\\
	\times \int_0^1 t^{a_{n+1}}(1-t)^{b_{n}-a_{n+1}-1}  {_{n}F_{n-1}}\left(\begin{array}{cccc}
		a_1 & a_2&\ldots  & a_{n} \\
		& b_1&\ldots& b_{n-1}
	\end{array}|zt
	\right)\frac{dt}{t},
	\end{align}
	where $\Gamma(z)$ is the usual gamma function. This integral representation plays a key role  for associating a modular form to a hypergeometric datum.
		\par Next, we recall  the Gross-Koblitz formula which relates Gauss sums to the $p$-adic gamma function.
	\subsection{$p$-Adic preliminaries and Gross-Koblitz formula}
	Let  $\mathbb{Q}_p$ denote the field of $p$-adic numbers.
	Let $\overline{\mathbb{Q}_p}$ be the algebraic closure of $\mathbb{Q}_p$ and $\mathbb{C}_p$ be the completion of $\overline{\mathbb{Q}_p}$.
	Let $\mathbb{Z}_q$ be the ring of integers in the unique unramified extension of $\mathbb{Q}_p$ with residue field $\mathbb{F}_q$.
	We know that $\chi\in \widehat{\mathbb{F}_q^{\times}}$ takes values in $\mu_{q-1}$, where $\mu_{q-1}$ is the group of all the $(q-1)$-th roots of unity in $\mathbb{C}^{\times}$. Since $\mathbb{Z}_q^{\times}$ contains all the $(q-1)$-th roots of unity,
	we can consider multiplicative characters on $\mathbb{F}_q^\times$
	to be maps $\chi: \mathbb{F}_q^{\times} \rightarrow \mathbb{Z}_q^{\times}$.
	Let $\omega_q: \mathbb{F}_q^\times \rightarrow \mathbb{Z}_q^{\times}$ be the Teichm\"{u}ller character.
	For $a\in\mathbb{F}_q^\times$, the value $\omega_q(a)$ is just the $(q-1)$-th root of unity in $\mathbb{Z}_q$ such that $\omega_q(a)\equiv a \pmod{p}$.
	\par Now, we introduce the Gauss sum. Let $\zeta_p$ be a fixed primitive $p$-th root of unity
	in $\overline{\mathbb{Q}_p}$. The trace map $\text{tr}: \mathbb{F}_q \rightarrow \mathbb{F}_p$ is given by
	\begin{align}
		\text{tr}(\alpha)=\alpha + \alpha^p + \alpha^{p^2}+ \cdots + \alpha^{p^{r-1}}.\notag
	\end{align}
	Then the additive character
	$\theta: \mathbb{F}_q \rightarrow \mathbb{Q}_p(\zeta_p)$ is defined by
	\begin{align}
		\theta(\alpha)=\zeta_p^{\text{tr}(\alpha)}.\notag
	\end{align}
	For $\chi \in \widehat{\mathbb{F}_q^\times}$, the \emph{Gauss sum} is defined by
	\begin{align}
		g(\chi):=\sum\limits_{x\in \mathbb{F}_q}\chi(x)\theta(x) .\notag
	\end{align}
	For $A, B\in\widehat{\mathbb{F}_q^{\times}}$ such that $AB\neq\varepsilon$. Then 
	\begin{align*}
		J(A, B)=\frac{g(A)g(B)}{g(AB)}.
	\end{align*}
	Now, we recall the $p$-adic gamma function. For a positive integer $n$, the $p$-adic gamma function $\Gamma_p(n)$ is defined as
	\begin{align}
		\Gamma_p(n):=(-1)^n\prod\limits_{0<j<n,p\nmid j}j\notag
	\end{align}
	and one extends it to all $x\in\mathbb{Z}_p$ by setting $\Gamma_p(0):=1$ and
	\begin{align}
		\Gamma_p(x):=\lim_{x_n\rightarrow x}\Gamma_p(x_n)\notag
	\end{align}
	for $x\neq0$, where $x_n$ runs through any sequence of positive integers $p$-adically approaching $x$.
	This limit exists, is independent of how $x_n$ approaches $x$,
	and determines a continuous function on $\mathbb{Z}_p$ with values in $\mathbb{Z}_p^{\times}$.
	Let $\pi \in \mathbb{C}_p$ be the fixed root of $x^{p-1} + p=0$ which satisfies
	$\pi \equiv \zeta_p-1 \pmod{(\zeta_p-1)^2}$. Then the Gross-Koblitz formula relates Gauss sums and the $p$-adic gamma function as follows.
	\begin{theorem}\emph{\cite[Gross-Koblitz]{gross}.}\label{gross} Let $p$ be a prime and  $a\in \mathbb{Z}$ such that $0\leq a\leq p-2$. Then
		\begin{align}
			g(\overline{\omega_p}^a)=-\pi^{a}\Gamma_p\left( \frac{a}{p-1} \right).\notag
		\end{align}
	\end{theorem}
   \subsection{Commutative Formal Group Law (CFGL)}
	Now we recall the commutative formal group law property, which will be used in the proof of Theorem~\ref{main-thm-K5}. 
	\begin{proposition}[\cite{beukers},\cite{Stienstra_Beukers}]\label{CFGL-isom}
	Let $p$ be a prime and $R$ be a $\mathbb{Z}_p$-algebra equipped with an endomorphism $\sigma:R\rightarrow R$ satisfying that $\sigma(a)-a^p\in pR$ for all $a \in R$. Let $\omega(x) = \sum_{n=1}^{\infty} b_{n}x^{n-1} {dx}$ with $b_{n} \in R$ for all $n \geq 1$. Let $x(u) = \sum_{n=1}^{\infty} a_{n}u^{n}\in R[[u]]$ and suppose $\omega(x(u)) = \sum_{n=1}^{\infty} c_{n}u^{n-1} {du}$ with $c_n\in R$. 		
	If there exists $\alpha_{p}, \beta_{p} \in R$ with $\beta_{p}\in pR$ such that for all $m,r \in \mathbb{N}$,
	\begin{equation}\label{eq_b_congruence}
					b_{mp^{r}}-\alpha_{p}\sigma(b_{mp^{r-1}}) +\beta_{p} \sigma^2(b_{mp^{r-2}}) \equiv 0 \pmod{p^r};
	\end{equation}
	then for all $m,r \in \mathbb{N}$
	\begin{equation}\label{eq_c_congruence}
					c_{mp^{r}} -\alpha_{p}\sigma(c_{mp^{r-1}}) + \beta_{p}\sigma^2(c_{mp^{r-2}}) \equiv 0 \pmod{p^r}.
    \end{equation}   
	If $a_{1}$ is invertible in $R$ then \eqref{eq_c_congruence} implies \eqref{eq_b_congruence}. Moreover, if $b_p/b_1\in R^\times$, then there exists $\mu_p\in  R^\times$ such that for $m,r\ge 1$
	\begin{equation*}
	b_{mp^r}\equiv \mu_p \sigma(b_{mp^{r-1}})\pmod {p^r},\quad \text{and} \quad  c_{mp^r}\equiv \mu_p \sigma(c_{mp^{r-1}})\pmod {p^r}.
	\end{equation*}
	Here $\mu_p$ is referred to as the unit root of $\omega(x)$.
	\end{proposition}
	\subsection{The Explicit Hypergeometric Modularity Method (EHMM)}\label{sec:EHMM}
	We now recall the EHMM, as developed by Allen \emph{et al.} \cite{allen_EHMM2,allen_EHMM1}. 
	\begin{theorem}\emph{\cite[Theorem 2.1]{allen_EHMM1}.}\label{thm-EHMM}
	Let  $n=3$ or $4$. Assume $\alpha^\flat=\{r_1,\cdots,r_{n-1}\}$, where $0<r_1
	\le\cdots\le r_{n-1}<1$, and $\beta^\flat=\{1,\cdots,1\}$ $($with multiplicity $n-1$$)$, with $r_n,q_n$ such that  $0<r_n<q_n\le 1$ and $r_2<q_n$.  Let $\mathrm{HD}=\{\{r_n\}\cup \alpha^\flat,\{q_n\}\cup \beta^\flat\}$ and $M=\mathrm{lcd}(\mathrm{HD})$. 
	Further assume $\gamma(\mathrm{HD}):=-1+\sum_{i=1}^{n}(q_i-r_i) \leq 1$ and that
	\begin{enumerate}
	\item there exists a modular function $t=C_1 q^{}+O(q^2) \in\mathbb{Z}[[q]]$ such that 
	\begin{equation*}
				f_{\mathrm{HD}}(q):= C_1^{-r_n}\cdot t(q)^{r_n}(1-t(q))^{q_n-r_n-1}F(\alpha^\flat,\beta^\flat;t(q))q\frac{dt(q)}{t(q)dq}
	\end{equation*}
	is a congruence weight-$n$ holomorphic cusp form satisfying that, for each prime $p$ such that $p\equiv 1\pmod M$, $T_p f_{\mathrm{HD}}=\tilde b_p  f_{\mathrm{HD}}$ for some $\tilde b_p$  in $\mathbb{Z}$  where $T_p$ is the $p$th Hecke operator;
	\item for any prime ideal $\mathfrak{p}$ in $\mathbb{Z}[\zeta_M]$ above $p$ such that $p\equiv 1\pmod M$,  
	$$\iota_{\mathfrak{p}}(r_n)(C_1)^{-1}\prod_{i=2}^{n-1}\iota_{\mathfrak{p}}(r_i)(-1)   \cdot \mathbb{P}\left(\mathrm{HD};1;{\mathfrak{p}}\right)\in \mathbb{Z}.$$  
	\end{enumerate}
	Then there exists a normalized Hecke eigenform $f_{\mathrm{HD}}^\sharp$ built from $ f_{\mathrm{HD}}$, not necessarily unique, such that for each $p>29$ such that $p\equiv 1 \pmod M$, $\tilde b_p=a_p(f_{\mathrm{HD}}^\sharp)$. More explicitly,  
	\begin{align*}
			a_p(f_{\mathrm{HD}}^\sharp)=&   { (-1)^{n-1}} \iota_{\mathfrak{p}}(r_n)(C_1)^{-1}  \cdot \prod_{i=2}^{n-1}\iota_{\mathfrak{p}}(r_i)(-1)\cdot \mathbb{P}(\mathrm{HD};1;{\mathfrak{p}})-\delta_{\gamma(\mathrm{HD})=1}\cdot \psi_{\mathrm{HD}}(p) \cdot p.
	\end{align*} 
	Here $\delta_{\gamma(\mathrm{HD})=1}$ is equal to 1 when ${\gamma(\mathrm{HD})=1}$ and is 0 otherwise and
	\begin{equation*}
			\psi_{\mathrm{HD}}(p) \equiv (-1)^{n-1}\cdot{C_1^{(p-1)r_n}}\frac{\Gamma_p\left({ q_n-r_n}\right)}{\Gamma_p{(r_1)}\cdots\Gamma_p(r_{n-1})}\pmod p. 
	\end{equation*}
	In terms of Galois representations,
	$$\rho_{f_{\mathrm{HD}}^{\sharp}}|_{G(M)} \cong \chi_{\mathrm{HD}} \otimes \rho_{\{\mathrm{HD};1\}}-\delta_{\gamma(\mathrm{HD})=1}\cdot \psi_{\mathrm{HD}}\epsilon_\ell|_{G(M)},$$
	where $G(M)$ denotes the Galois group $\mathrm{Gal}\left({\overline{\mathbb{Q}}}/{\mathbb{Q}(\zeta_{M})}\right)$, $\rho_{f_{\mathrm{HD}}^{\sharp}}$ is the Galois representation associated to the eigenform $f_{\mathrm{HD}}^{\sharp}$ by Deligne, $\rho_{\{\mathrm{HD};1\}}$ is the hypergeometric Galois representation associated to $\mathrm{HD}$ from Theorem~\ref{thm:Katz} by Katz, $\epsilon_\ell$ is the $\ell$-adic cyclotomic character, and $$\chi_{\mathrm{HD}}(\mathfrak{p}):= \iota_{\mathfrak{p}}(r_{n})(C_{1})^{-1} \cdot \prod_{i=1}^{n-1} \iota_{\mathfrak{p}}(r_{i})(-1).$$
	\end{theorem}
	We now discuss hypergeometric character sums and condition (2) of Theorem~\ref{thm-EHMM} in preparation for proving Theorem~\ref{main-thm-K4}.
	\subsection{Hypergeometric Character Sums and Extendability}
	We recall the fundamental results of Katz  on hypergeometric Galois representations.
	Let $G(M):=\mathrm{Gal}\left({\overline{\mathbb{Q}}}/{\mathbb{Q}(\zeta_{M})}\right)$.
	\begin{theorem}[\cite{katz_galois1}]\label{thm:Katz}
	Let $\ell$ be a prime. Let $\mathrm{HD}=\{\{r_1,\dots,r_n\},\{q_1=1,q_2,\dots,q_n\}\}$
	be a primitive hypergeometric datum with $M=\mathrm{lcd}(\mathrm{HD})$.
	For any $\lambda\in\mathbb{Z}[\zeta_M,1/M]\setminus\{0\}$, the following hold:
	\begin{itemize}
	\item[(i)] 
	There exists an $\ell$-adic Galois representation
    $\rho_{\{\mathrm{HD};\lambda\}}:G(M)\to GL(W_\lambda),$ unramified almost everywhere, such that for every nonzero prime ideal $\mathfrak{p}$ of $\mathbb{Z}[\zeta_M,1/(M\ell\lambda)]$ of norm $N(\mathfrak{p})=|\mathbb{Z}[\zeta_M]/\mathfrak{p}|$,
	\[\mathrm{Tr}\,\rho_{\{\mathrm{HD};\lambda\}}(\mathrm{Frob}_{\mathfrak p}) = (-1)^{n-1}\,\iota_{\mathfrak p}(r_1)(-1)\,
			\mathbb{P}(\mathrm{HD};1/\lambda;\mathfrak p),
	\]
	where $\mathrm{Frob}_{\mathfrak p}$ denotes the geometric Frobenius conjugacy class of $G(M)$ at $\mathfrak{p}$.
	\item[(ii)] 
	If $\lambda=1$, then
	$\dim_{\overline{\mathbb Q}_\ell}W_1=n-1$. All roots of the characteristic polynomial of $\rho_{\{\mathrm{HD};1\}}(\mathrm{Frob}_{\mathfrak p})$ have absolute value less than or equal to
	$ N(\mathfrak p)^{(n-1)/2}$ under all archimedean embeddings.
	\end{itemize}
	 \end{theorem}
	 Observe that condition (2) of Theorem~\ref{thm-EHMM} is  equivalent  to the statement that the 
	 $G(M)$-representation $ \chi_{\mathrm{HD}} \otimes \rho_{\{\mathrm{HD};1\}}-\delta_{\gamma(\mathrm{HD})=1}\cdot \psi_{\mathrm{HD}}\epsilon_\ell|_{G(M)}$ can be extended to the full Galois group $G_\mathbb{Q}:=\mathrm{Gal}\left(\overline{\mathbb{Q}}/\mathbb{Q}\right)$. This reduction is a consequence of the following standard extendability criterion. 
	\begin{proposition}\emph{\cite[Proposition~ 4.2]{allen_EHMM1}.}\label{extendableprop}
	Let $M$ be a positive integer. Assume $\rho$ is a semi-simple finite dimensional $\ell$-adic representation of $G(M)$ which is isomorphic to $\rho^{\tau}$ for each $\tau \in \textrm{G}_{\mathbb{Q}}$, then $\rho$ is extendable to $\textrm{G}_{\mathbb{Q}}$. Equivalently, for each nonzero prime ideal $\mathfrak{p}$ of $\mathbb{Z}[\zeta_{M}]$ unramified for $\rho$, we have $\mathrm{Tr} \, \rho(\textnormal{Frob}_{\mathfrak{p}}) \in \mathbb{Z}$.
	\end{proposition}
	 Since the traces of Katz representations are described in terms of the $\mathbb{P}$-function,  in certain cases condition (2) of Theorem~\ref{thm-EHMM} reduces to establishing appropriate character sum identities among the functions
	$\mathbb{P}(\mathrm{HD};1;\mathfrak p)$.
	Now we recall one such transformation formula due to Li \emph{et al.} \cite{whipple}.
	\begin{lemma}\emph{\cite[Proposition~ 1]{whipple}.}\label{trans_P_function}
    Let $\mathbb F_q$ be a finite field of characteristic $p>2$, let $A_i,B_j\in\widehat{\mathbb F_q^\times}$, and let $t\neq0$. Then
	\begin{align*}
			&{_{n+1}}\mathbb P_n\!\left(
			\begin{array}{cccc}
				A_1 & A_2 & \dots & A_{n+1} \\
				& B_2 & \dots & B_{n+1}
			\end{array}
			| \frac{1}{t}\right)\notag\\
			&=
			A_1(-t)\!\left(\prod_{i=2}^{n+1}A_iB_i(-1)\right)\cdot
			{_{n+1}}\mathbb {P}_n\left(
			\begin{array}{cccc}
				A_1 & A_1\overline{B_2} & \dots & A_1\overline{B_{n+1}} \\
				& A_1\overline{A_2} & \dots & A_1\overline{A_{n+1}}
			\end{array}
			| t\right).
	\end{align*}
	\end{lemma}
	\section{Construction of $\mathbb{K}_4$ and $\mathbb{K}_5$  functions}
	In this section, we construct two important classes of eta-quotients arising from EHMM,  analogous to the $\mathbb{K}_1, \mathbb{K}_2$, and  $\mathbb{K}_3$ functions studied in \cite{allen_EHMM2, allen_EHMM1, grove, rosen2}. We denote these new families by $\mathbb{K}_4$ and $\mathbb{K}_5$. They play a key role in the construction of the Hecke eigenforms $f_{\mathrm{HD}_{\mathbb{K}_{4}\left(r\right)}}^{\sharp}$  and $f_{\mathrm{HD}_{\mathbb{K}_{5}\left(r\right)}}^{\sharp}$,  which appear in Theorem~\ref{main-thm-K4} and Theorem~\ref{main-thm-K5}, respectively.
	\subsection{The $\mathbb{K}_4$-functions}
	Consider the hypergeometric data
	\begin{equation}\label{k4data}
		\mathrm{HD}_{\mathbb{K}_{4}\left(r\right)}: = \{\{1/2,1/2,1/2,r\},\{1,1,1,r+{1}/{2}\}\},
	\end{equation}
	where $r \in \mathbb{Q}$.
   Now if $r >0$, using the Euler integral formula \eqref{euler}, we obtain
	\begin{align}\label{4F3integralK4}
			&{_{4}}F_{3}\left(\begin{array}{cccc}
				\frac{1}{2}, & \frac{1}{2}, & \frac{1}{2}, & r \\
				& 1, & 1, & r+\frac{1}{2}
			\end{array}\mid 1
			\right)\notag\\
			&= \frac{\Gamma(r+1/2)}{\Gamma(r)\Gamma(1/2)} \int_{0}^{1} t^{r-1}(1-t)^{-1/2} \cdot {_{3}}F_{2}\left(\begin{array}{cccc}
				\frac{1}{2},  & \frac{1}{2}, & \frac{1}{2}\\
				& 1, & 1
			\end{array}\mid t
			\right) \, {dt}.
	\end{align}
	 We next compute the differential of \eqref{4F3integralK4} at a suitable  Hauptmodul. The Schwarz map \cite[Section 3.2.2]{FL} and the Clausen formula \cite[(7.1)]{FL} connect  the congruence subgroup $\Gamma_0(2)$ and the $_{3}F_{2}(\{1/2,1/2,1/2\},\{1,1\};t)$ series. Consequently, we are led to the modular function
	\begin{equation}\label{hauptmodul}
	t_2(\tau) = -64 \frac{\eta(2 \tau)^{24}}{\eta(\tau)^{24}},
    \end{equation}
    where 
   \begin{equation*}
   	\eta(\tau):=q^{1/24}\prod_{n=1}^{\infty}(1-q^n)
   \end{equation*}
   is the Dedekind-eta function, $q=e^{2\pi i \tau}$ and $\tau$ lies in the complex upper-half plane  $\mathcal{H} := \{\tau \in \mathbb{C} \, | \, \textrm{Im}(\tau) > 0\}$.
   It is known that $t_2(\tau)$ serves as a Hauptmodul for the congruence subgroup $\Gamma_0(2)$. Now we recall the three weight $1/2$ Jacobi theta functions from \cite{borwein1, zagier}.
   \begin{align*}
		\theta_3(\tau)=\frac{\eta(\tau)^5}{\eta(\tau/2)^2\eta(2\tau)^2},\
		\theta_4(\tau)=\frac{\eta(\tau/2)^2}{\eta(\tau)},\
		\theta_2(\tau)=\frac{2\eta(2\tau)^2}{\eta(\tau)}.
	\end{align*}
	Using the explicit representation of above Jacobi theta series as eta-quotients in \eqref{hauptmodul}, we obtain
	\begin{equation}\label{hauptmodul_theta}
	t_2(\tau) =-\frac{\theta_2(\tau)^8}{4\theta_3(\tau)^4\theta_4(\tau)^4}.	
	\end{equation}
 Recall the following identities from \cite[{Appendix A}]{allen_EHMM1}:
  \begin{equation}\label{3F2_hauptmodul}
	_3F_2\left(\begin{array}{cccc}
		\frac{1}{2} & \frac{1}{2}  & \frac{1}{2} \\
		&1 & 1
	\end{array}|t_2(\tau)
	\right)=\theta_4(2\tau)^4, 
  \end{equation}
  \text{and}
   \begin{equation}\label{derivative}
	q \frac{d t_2(\tau)}{t_2(\tau) dq}=(1-t_2(\tau))^{1/2}\theta_4(2\tau)^4.
  \end{equation}
  Now evaluating the differential of \eqref{4F3integralK4} at $t = t_2(\tau)$, and using \eqref{hauptmodul_theta}, \eqref{3F2_hauptmodul}, and \eqref{derivative}, we obtain
  \begin{equation}\label{K4eval}
		t_2(\tau)^{r}(1-t_2(\tau))^{-1/2} \cdot \, _3F_2\left(\begin{array}{cccc}
			\frac{1}{2} & \frac{1}{2}  & \frac{1}{2} \\
		&1 & 1
		\end{array}|t_2(\tau)
		\right) q \frac{d t_2(\tau)}{t_2(\tau) dq} = (-64)^r \mathbb{K}_{4}\left(r\right)(\tau),
	\end{equation}
	where the $\mathbb{K}_{4}$ functions are defined as 
	\begin{equation*}
		\mathbb{K}_{4}\left(r\right)(\tau) := {\eta(2 \tau)^{24r-8}}{\eta(\tau)^{-24r+16}}. 
	\end{equation*}
	 We restrict $r$ to a suitable subset of $\mathbb{Q}$ in order to ensure the convergence of the Euler integral formula and obtain congruence holomorphic cusp forms. Specifically, we define 
	\begin{equation*}
		\mathbb{S}_{4} := \{r \in \mathbb{Q} \, | \, 0 < r < 1, 24r \in \mathbb{Z} \}.
	\end{equation*}
	If $r=a/b$ is in reduced form, we set $N_{\mathbb{K}_{4}}(r):=b$. Then, by a well-known classical result (see, for example, \cite{ono_modularity}), the functions $\mathbb{K}_{4}(r)\left(N_{\mathbb{K}_{4}}(r)\tau\right)$ for $r \in \mathbb{S}_{4}$ are congruence holomorphic cuspforms of weight four. The  levels and the associated characters for the family $\mathbb{K}_4$  are listed in Table~\ref{K4table}.
	\subsection{The $\mathbb{K}_5$-functions}
   Similar to the hypergeometric data $\mathrm{HD}_{\mathbb{K}_{4}\left(r\right)}$, we now consider the hypergeometric  data
   \begin{equation*}
		\mathrm{HD}_{\mathbb{K}_{5}\left(r\right)}: = \{\{1/3,r\},\{1,1\}\},
	\end{equation*}
	where $r \in \mathbb{Q}$. 
   Now if $0<r<1 $, using the Euler integral formula \eqref{euler}, we have
   \begin{align}\label{2F1integralK5}
		{_{2}}F_{1}\left(\begin{array}{cccc}
			\frac{1}{3}, & r \\
			& 1
		\end{array}\mid 1
		\right)= \frac{1}{\Gamma(r)\Gamma(1-r)} \int_{0}^{1} t^{r-1}(1-t)^{-r-1/3} {dt}.
	\end{align}
	We consider the modular function
	\begin{equation*}
		t_3(\tau) = 27 \frac{\eta(3 \tau)^{9}}{\left(3\eta(3\tau)^{3}+\eta\left(\frac{\tau}{3}\right)^3\right)^3},
	\end{equation*}
	which serves as a Hauptmodul for the congruence subgroup $\Gamma_0(3)$. Now we recall the cubic analogues of the  weight $1/2$ Jacobi theta functions, namely $a(\tau), b(\tau),$ and $c(\tau)$, as introduced in \cite{borwein}.
	\begin{align*}
		a(\tau)&=\frac{3\eta(3\tau)^{3}+\eta\left(\frac{\tau}{3}\right)^3}{\eta(\tau)},\\
		b(\tau)&=\frac{\eta(\tau)^3}{\eta(3\tau)},\\
		c(\tau)&=\frac{3\eta(3\tau)^3}{\eta(\tau)}.
	\end{align*}
	 From \cite{grove}, we have
	\begin{equation}\label{derivative-2}
		t_3(\tau)=\frac{c^3(\tau)}{a^3(\tau)},\ 1-t_3(\tau)=\frac{b^3(\tau)}{a^3(\tau)}, \ \text{and}\	q \frac{d t_3(\tau)}{t_3(\tau) dq}=\frac{b^3(\tau)}{a(\tau)}.
	\end{equation}
	Now evaluating the differential of \eqref{2F1integralK5} at $t = t_3(\tau)$ and using  \eqref{derivative-2}, in analogy with \eqref{K4eval}, we deduce
	\begin{equation*}
		t_3(\tau)^{r}(1-t_3(\tau))^{-r-1/3}  q \frac{d t_3(\tau)}{t_3(\tau) dq} = {27}^r\cdot \mathbb{K}_{5}\left(r\right)(\tau),
	\end{equation*}
	where the $\mathbb{K}_{5}$ functions are defined as 
	\begin{equation*}
		\mathbb{K}_{5}\left(r\right)(\tau) := {\eta(3 \tau)^{12r-2}}{\eta(\tau)^{-12r+6}}.
	\end{equation*}
	Analogous to $\mathbb{S}_4$, we define  a subset $\mathbb{S}_5$ of $\mathbb{Q}$ by
	\begin{equation*}
		\mathbb{S}_{5} := \{r \in \mathbb{Q} \, | \, 0 < r < 2/3, 12r \in \mathbb{Z} \}.
	\end{equation*}
	The restriction $r\in\mathbb{S}_{5}$ ensures the convergence of the Euler integral representation and yields congruence holomorphic cusp forms. For $r=a/b\in\mathbb{S}_{5}$, where $a/b$ is in lowest terms, we define  $N_{\mathbb{K}_{5}}(r):=b$. Then, by a well-known classical result (see, for example, \cite{ono_modularity}), the functions  $\mathbb{K}_{5}(r)\left(N_{\mathbb{K}_{5}}(r)\tau\right)$ for $r \in \mathbb{S}_{5}$   are congruence holomorphic cuspforms of weight  two. 
  Now, using \eqref{euler} together with  the change of variable $t\mapsto 1-t$, we obtain
		\begin{align*}
		{_{2}}F_{1}\left(\begin{array}{cccc}
			 \frac{1}{3}, &\frac{2}{3}-r \\
			& 1
		\end{array}\mid 1
		\right)=\frac{\Gamma(r)\Gamma(1-r)}{\Gamma(1/3+r)\Gamma(2/3-r)}	{_{2}}F_{1}\left(\begin{array}{cccc}
			\frac{1}{3}, & r \\
			& 1
		\end{array}\mid 1
		\right), ~\text{for}~0<r<\frac{2}{3}.
	\end{align*}
Using this identity,  we consider the hypergeometric data $\left(\left\{\frac{1}{3},\frac{2}{3}-r\right\},\{1,1\}\right)$,
	which, analogous to the construction of $\mathbb{K}_{5}$, gives rise to a new family of eta-quotients, denoted by $\mathbb{K}_{5}'$, defined by
	\begin{align*}\label{K5_prime_eval}
		\mathbb{K}_{5}^\prime\left(r\right)(\tau) := {\eta( \tau)^{12r-2}}{\eta(3\tau)^{-12r+6}}.
	\end{align*}
	Since the families $\mathbb{K}_{5}$ and $\mathbb{K}_{5}^{\prime }$ are related via the Fricke involution, the EHMM construction yields exactly four distinct families for the function $\mathbb{K}_{5}$, which are summarized in Table~\ref{K5table}.
	\renewcommand{\arraystretch}{1.2}
	\begin{table}[ht]
		\begin{center}
			\begin{tabular}{|c|c|c|c|c|c|}
				\hline
				\text{Families} & {$r = j/24$} & $N_{\mathbb{K}_{4}}(r)$ & \text{Eigenform(s)} & \text{Level} & \textrm{Character} \\ \hline
				$(1)$ & $1/2$ & $2$ & $f_{8.4.a.a} = \mathbb{K}_{4}\left(\frac{1}{2}\right)(2 \tau)$ & $8$ & \text{trivial} \\ \hline
				$(2)$ &$1/3,2/3$ & $3$ & $f_{9.4.a.a} = \mathbb{K}_{4}\left(\frac{1}{3}\right)(3 \tau)$ & $18$ &  \text{trivial}\\
				&&&$f_{9.4.a.a}(2 \tau) = \mathbb{K}_{4}\left(\frac{2}{3}\right)(3 \tau)$ && \\ \hline
				$(3)$ & $1/4,3/4$ & $4$ & $f_{32.4.a.a} = \mathbb{K}_{4}\left(\frac{1}{4}\right)(4 \tau)$ & $32$ & \text{trivial} \\ 
				&&&$-8 \mathbb{K}_{4}\left(\frac{3}{4}\right)(4 \tau)$ &&\\
				&&&$f_{32.4.a.c} = \mathbb{K}_{4}\left(\frac{1}{4}\right)(4 \tau)$&&\\
				&&&$+8 \mathbb{K}_{4}\left(\frac{3}{4}\right)(4 \tau)$ && \\ \hline
				$(4)$ & $1/6,5/6$ & $6$ & $f_{72.4.a.d} = \mathbb{K}_{4}\left(\frac{1}{6}\right)(6 \tau)$ & $72$&\text{trivial}  \\ 
				&&&$+16\mathbb{K}_{4}\left(\frac{5}{6}\right)(6 \tau)$ &&\\ 
				&&&$f_{72.4.a.a} = \mathbb{K}_{4}\left(\frac{1}{6}\right)(6 \tau)$&&\\
				&&&$-16\mathbb{K}_{4}\left(\frac{5}{6}\right)(6 \tau)$&&\\ \hline
				$(5)$&$1/8,3/8$&$8$&$f_{128.4.b.e}=\mathbb{K}_{4}\left(\frac{1}{8}\right)(8 \tau)$ &$128$ &$\chi(2)$\\
				&$5/8,7/8$& & $+2\sqrt{-10}\mathbb{K}_{4}\left(\frac{3}{8}\right)(8 \tau)$&&\\
				&&& $+8\sqrt{-5}\mathbb{K}_{4}\left(\frac{5}{8}\right)(8 \tau)$&&\\
				&&& $+16\sqrt{2}\mathbb{K}_{4}\left(\frac{7}{8}\right)(8 \tau)$&&\\\hline
				$(6)$ & $1/12,5/12$  & $12$ & $f_{288.4.a.l} = \mathbb{K}_{4}\left(\frac{1}{12}\right)(12 \tau)$& $288$ & \text{trivial}\\
				& $7/12,11/12$& & $+ 4\sqrt{3}  \mathbb{K}_{4}\left(\frac{5}{12}\right)(12 \tau)$ &&\\ 
				&&&$-8\sqrt{3}\mathbb{K}_{4}\left(\frac{7}{12}\right)(12 \tau)$ &&\\ 
				&&&$-32\mathbb{K}_{4}\left(\frac{11}{12}\right)(12 \tau)$ &&\\
				&&&$f_{288.4.a.m} = \mathbb{K}_{4}\left(\frac{1}{12}\right)(12 \tau)$&&\\
				&&&$+ 4\sqrt{3}  \mathbb{K}_{4}\left(\frac{5}{12}\right)(12 \tau)$ &&\\
				&&&$+8\sqrt{3}\mathbb{K}_{4}\left(\frac{7}{12}\right)(12 \tau)$ &&\\
				&&&$+32\mathbb{K}_{4}\left(\frac{11}{12}\right)(12 \tau)$ &&\\ \hline
				$(7)$ & $1/24,5/24$  & $24$ & $f_{1152.4.d.q} = \mathbb{K}_{4}\left(\frac{1}{24}\right)(24 \tau)$& $1152$ & $\chi(2)$\\
				& $7/24,11/24$& & $+2\sqrt{-35}  \mathbb{K}_{4}\left(\frac{5}{24}\right)(24 \tau)$ &&\\ 
				&$13/24,17/24$&&$+2\sqrt{110}\mathbb{K}_{4}\left(\frac{7}{24}\right)(24 \tau)$ &&\\ 
				&$19/24,23/24$&&$-4\sqrt{-154}\mathbb{K}_{4}\left(\frac{11}{24}\right)(24 \tau)$ &&\\
				&&&$+4\sqrt{-77}\mathbb{K}_{4}\left(\frac{13}{24}\right)(24 \tau)$ &&\\
				&&&$+16\sqrt{55}\mathbb{K}_{4}\left(\frac{17}{24}\right)(24 \tau)$ &&\\
				&&&$-16\sqrt{-70}\mathbb{K}_{4}\left(\frac{19}{24}\right)(24 \tau)$ &&\\
				&&&$+32\sqrt{2}\mathbb{K}_{4}\left(\frac{23}{24}\right)(24 \tau)$ &&\\ \hline
			\end{tabular}
		\end{center}
		\caption{The families of $\mathbb{K}_{4}$ functions. In the above, $\chi(2):=(2/\cdot)$ denotes the Legendre symbol.}\label{K4table}
	\end{table}
	\begin{table}[ht]
		\begin{center}
			\begin{tabular}{|c|c|c|c|c|c|}
				\hline
				\text{Families} & {$r = j/12$} & $N_{\mathbb{K}_{5}}(r)$ & \text{Eigenform(s)} & \text{Level} & \textrm{Character} \\ \hline
				$(1)$ & $1/2$ & $2$ & $f_{36.2.a.a} = \mathbb{K}_{5}\left(\frac{1}{2}\right)(2 \tau)$ & $36$ & \text{trivial} \\ \hline
				$(2)$ &$1/3$ & $3$ & $f_{27.2.a.a} = \mathbb{K}_{5}\left(\frac{1}{3}\right)(3 \tau)$ & $27$ &  \text{trivial}\\ \hline
				$(3)$ & $1/6$ & $6$ & $f_{36.2.a.a} = \mathbb{K}_{5}\left(\frac{1}{6}\right)(6 \tau)$ & $36$ & \text{trivial} \\  \hline
				$(4)$ & $1/12$, $7/12$ & $12$ & $f_{432.2.c.a} = \mathbb{K}_{5}\left(\frac{1}{12}\right)(12 \tau)$ & $432$ & $\chi(3)$\\
				&&&$-3\sqrt{-3}  \mathbb{K}_{5}\left(\frac{7}{12}\right)(12 \tau)$&&\\ \hline
				$(5)$ & $1/4$ & $4$ & $f_{144.2.c.a} = \mathbb{K}_{5}\left(\frac{1}{4}\right)(4\tau)$ & $144$ & $\chi(3)$\\
				&&&$+(3+3\sqrt{-2})  \mathbb{K}_{5}\left(\frac{5}{12}\right)(12 \tau)$&&\\ \hline
					$(6)$ & $5/12$ & $12$ & $f_{144.2.c.a} = \mathbb{K}_{5}\left(\frac{1}{4}\right)(4\tau)$ & $144$ & $\chi(3)$\\
				&&&$+(3+3\sqrt{-2})  \mathbb{K}_{5}\left(\frac{5}{12}\right)(12 \tau)$&&\\ \hline
					\end{tabular}
			\end{center}
		\caption{The families of $\mathbb{K}_{5}$ functions. In the above $\chi(3):=(3/\cdot)$ denotes the Legendre symbol.}\label{K5table}
  \end{table}
	\subsection{The Galois Families for the $\mathbb{K}_4$ and $\mathbb{K}_5$ Functions }
	We now consider the cases in which the $\mathbb{K}_4$ and $\mathbb{K}_5$ functions can be completed to Hecke eigenforms, in a manner analogous to the $\mathbb{K}_1$, $\mathbb{K}_2$, and $\mathbb{K}_3$ functions. The instances in which the $\mathbb{K}_2$, $\mathbb{K}_3$, and $\mathbb{K}_4$ functions admit such completions are referred to as {\it Galois} cases. This notion was  first introduced in \cite{allen_EHMM1} and further developed in \cite{grove,rosen1}. Following \cite{allen_EHMM1, grove}, we extend this notion to the $\mathbb{K}_4$ and $\mathbb{K}_5$ families. Let $j \in \{4,5\}$ be fixed. For $i=1,2$, define
	\[
	s_i :=
	\begin{cases}
		r_i + \tfrac{1}{2}, & \text{if } r_i \in \mathbb{S}_4,\\
		1, & \text{if } r_i \in \mathbb{S}_5.
	\end{cases}
	\]
	\begin{definition} 
	We say that the pairs $(r_1,s_1)$ and $(r_2,s_2)$, with $r_1,r_2 \in \mathbb{S}_j$, are \emph{conjugate} if there exists an integer $c$, coprime to the level of the functions $\mathbb{K}_j(r_1)\left(N_{\mathbb{K}_{j}}(r_1)\tau\right)$ and $\mathbb{K}_j(r_2)\left(N_{\mathbb{K}_{j}}(r_2)\tau\right)$, such that
	$r_1 - c r_2 \in \mathbb{Z}\ \text{and} \ s_1 - c s_2 \in \mathbb{Z}.$
	\end{definition}
	\begin{definition}
		We say that a pair $\{(r_1, s_1):r_1\in \mathbb{S}_j\}$ is {\it Galois} for a collection of modular forms $\mathbb{K}_j(r_1)\left(N_{\mathbb{K}_{j}}(r_1)\tau\right)$ if the modular forms $\mathbb{K}_j(r_{1}^\sigma)\left(N_{\mathbb{K}_{j}}(r_{1}^\sigma)\tau\right)$ lie in the same Hecke orbit, for all conjugates $(r_{1}^\sigma, s_{1}^\sigma)$ of $(r_1,s_1)$. A collection of {\it Galois} pairs is called  a {\it Galois} family.
	\end{definition}
	We next consider the \emph{Galois} families for the $\mathbb{K}_4$ and $\mathbb{K}_5$ functions. Among the entries in Table~\ref{K4table}, the families $(1), (3), (4), (5), (6),$ and $(7)$ constitute the six \emph{Galois} families for the $\mathbb{K}_4$ functions. In contrast, family $(2)$ in Table~\ref{K4table}, namely, the functions $\mathbb{K}_4(j/3)(3\tau)$ for $j=1,2$ does not form a {\it Galois} family. For the $\mathbb{K}_5$ functions, each family listed in Table~\ref{K5table}, with the exception of $(4)$ and $(6)$, forms a \emph{Galois} family.
	\par
	We now describe how the eigenform is obtained by computing the relevant Hecke operators for the $\mathbb{K}_4(r)(\tau)$ and $\mathbb{K}_5(r)(\tau)$ functions. In particular, we discuss this explicitly for the family $\mathbb{K}_4(1/8)(8\tau)$.
	\begin{example}\label{TpforK4example}
		The eigenform $f_{{\mathrm{HD}_{\mathbb{K}_4(1/8)}}}^{\sharp}$ arises as a suitable linear combination of the $\mathbb{K}_4$ functions listed in Table~\ref{K4table}. In particular, the coefficients in this linear combination can be determined by analyzing the action of the Hecke operators $T_p$ on the cusp forms belonging to family $(5)$ in Table~\ref{K4table}. Now, we consider  the subspace of  
		$S_{4}(\Gamma_{0}(128), \chi(2))$  spanned by the functions
		$f_j := \mathbb{K}_{4}\left(\tfrac{j}{8}\right)(8\tau) \in q^{j}(1 + \mathbb{Z}[[q^{8}]])
		\ \text{for } j = 1,3,5,7.$ To construct the eigenform $f_{128.4.b.e}$, we compute the action of the Hecke operators $T_p$ for $p = 2,3,5,7$ on these basis elements. A direct calculation shows that $T_2(f_j) = 0 \ \text{for all } j = 1,3,5,7.$
		The nontrivial actions of $T_3$, $T_5$, and $T_7$ are given as follows:
		\begin{align*}
			T_3(f_1)&=-40f_3, T_3(f_3)=f_1, T_3(f_5)=-8f_7, T_3(f_7)=5f_5;\\
			T_5(f_1)&=-320f_5, T_5(f_3)=-64f_7, T_5(f_5)=f_1, T_5(f_7)=5f_3;\\
			T_7(f_1)&=512f_7, T_7(f_3)=-64f_5, T_7(f_5)=-8f_3, T_7(f_7)=f_1.
		\end{align*}
		These computations yield several relations among the Hecke operators $T_p$ for $p \in \{3,5,7\}$, namely
		\begin{equation}\label{HeckerelationsK4ex}
			T_3^2 = -40, \quad T_5^2 = -320, \quad T_3T_5 = 5T_7.
		\end{equation}
		According to the results of EHMM, one expects a normalized Hecke eigenform of the form
		\[
		f_{{\mathrm{HD}_{\mathbb{K}_4\left(\frac{1}{8}\right)}}}^{\sharp}
		= \sum_{j \in (\mathbb{Z}/8\mathbb{Z})^\times} d_j \cdot \mathbb{K}_{4}\left(\tfrac{j}{8}\right)(8\tau),
		\]
		for constants $d_1, d_3, d_5, d_7$. The normalization condition imposes $d_1 = 1$, while the remaining coefficients are determined by the Hecke actions of $T_3$, $T_5$, and $T_7$. In particular, from \eqref{HeckerelationsK4ex}, these constants satisfy
		$d_3^2 = -40, \ d_5^2 = -320$,  $5d_7 = d_3d_5.$ Note that the eigenform $f_{{\mathrm{HD}_{\mathbb{K}_4(1/8)}}}^{\sharp}$ is not uniquely determined, as different choices of $d_3$ and $d_5$ are possible. For a suitable choice of these constants, the resulting eigenform agrees with the data presented in Table~\ref{K4table}. In this case, it corresponds to the modular form labeled $f_{128.4.b.e}$ in the LMFDB.
	\end{example}
	Similar computations yield all eigenform completions associated with the $\mathbb{K}_4$ and $\mathbb{K}_5$ functions listed in Table~\ref{K4table} and Table~\ref{K5table}, respectively.  For example, the non-trivial action  of the Hecke operator corresponding to the family $(3)$  in Table~\ref{K4table} is given by $T_3(f_1)=64f_3, T_3(f_3)=f_1$, where $f_j= \mathbb{K}_4(j/4)(4\tau)\in	S_{4}(\Gamma_{0}(32))$ for $j=1,3$. Similarly, for the family $(4)$  in Table~\ref{K4table}, the Hecke operator satisfies $T_5(f_1)=256f_5, T_5(f_5)=f_1$, where $f_j= \mathbb{K}_4(j/6)(6\tau)\in	S_{4}(\Gamma_{0}(72))$ for $j=1,5$. For the family $(4)$  in Table~\ref{K5table}, the Hecke operator action is $T_7(f_1)=-27f_7, T_7(f_7)=f_1$, where $f_j= \mathbb{K}_5(j/12)(12\tau)\in	S_{2}(\Gamma_{0}(432), \chi(3))$ for $j=1,7$. For the families $(5)$ and $(6)$  in Table~\ref{K5table}, the action of Hecke operator is given by $T_5(f_1)=-3f_1-27f_5, T_5(f_5)=f_1+3f_5$, where $f_1=\mathbb{K}_5(1/4)(4\tau), f_5= \mathbb{K}_5(5/12)(12\tau)\in	S_{2}(\Gamma_{0}(144), \chi(3))$. The constants appearing in the linear combinations of $\mathbb{K}_4$ functions for families  $(6)$ and $(7)$ are determined via analogous Hecke operator computations,  as recorded in Table~\ref{table5} and Table ~\ref{table6}, respectively. 
	\begin{table}[ht]
		\begin{center}
			\begin{tabular} {|c|c|c|c|c|}
		\hline
		&$f_1$&$f_5$&$f_7$&$f_{11}$      \\\hline 
		$T_5$&$208f_{11}$&$13f_{7}$&$16f_{5}$&$f_{1}$  \\\hline
		$T_7$&$832f_{7}$&$13f_{11}$&$f_{1}$&$64f_{5}$  \\\hline
		$T_{11}$&$1024f_{5}$&$f_{1}$&$16f_{11}$&$64f_{7}$  \\\hline
	\end{tabular}
   \end{center}
    \caption{The action of $T_5, T_7$, and $T_{11}$ on the subspace of $S_4(\Gamma_0(288))$ spanned by the functions $f_j=\mathbb{K}_4(j/12)(12\tau)$ for $j=1,5, 7, 11$.}\label{table5}
    \end{table}
	\begin{table}[ht]
		\begin{center}
			\begin{tabular} {|c|c|c|c|c|c|c|c|c|}
		\hline
		&$f_1$&$f_5$&$f_7$&$f_{11}$&$f_{13}$&$f_{17}$&$f_{19}$&$f_{23}$      \\\hline 
		$T_5$&$-140f_{5}$&$f_{1}$&$28f_{11}$&$-5f_{7}$&$20f_{17}$&$-7f_{13}$&$-4f_{23}$&$35f_{19}$  \\\hline
		$T_7$&$440f_{7}$&$-88f_{11}$&$f_{1}$&$-5f_{5}$&$-40f_{19} $&$8f_{23}$&$-11f_{13}$&$55f_{17}$ \\\hline
		$T_{11}$&$-2464f_{11}$&$-88f_{7}$&$28f_{5}$&$f_{1}$&$-32f_{23}$ &$-56f_{19}$&$44f_{17}$&$77f_{13}$ \\\hline
		$T_{13}$&$-4928f_{13}$&$704f_{17}$&$448f_{19}$&$-64f_{23}$&$f_1$&$-7f_5$&$-11f_7$&$77f_{11}$\\\hline
		$T_{17}$&$14080f_{17}$&$704f_{13}$&$256f_{23}$&$320f_{19}$&$20f_5$&$f_1$&$44f_{11}$&$55f_{7}$\\\hline
		$T_{19}$&$-17920f_{19}$&$-512f_{23}$&$448f_{13}$&$320f_{17}$&$-40f_7$&$-56f_{11}$&$f_1$&$35f_5$\\\hline
		$T_{23}$&$2048f_{23}$&$-512f_{19}$&$256f_{17}$&$-64f_{13}$&$-32f_{11}$&$8f_{7}$&$-4f_5$&$f_1$\\\hline
	\end{tabular}
	\end{center}
	\caption{The action of $T_5, T_7, T_{11}, T_{13}, T_{17}, T_{19}$, and $T_{23}$ on the subspace of $S_4(\Gamma_0(1152),\chi(2))$ spanned by the functions $f_j=\mathbb{K}_4(j/24)(24\tau)$ for $j=1,5, 7, 11, 13, 17, 19, 23$.}\label{table6}
    \end{table}
   \section{Proof of Theorem~\ref{main-thm-K4} and \ref{main-thm-K5}}
    Before proving Theorem~\ref{main-thm-K4}, we first prove a proposition.
    \begin{proposition}\label{prop-main-thm}
	If $r\in\{1/8, 1/12, 1/24\}$, then for any prime ideal $\mathfrak{p}\in \mathbb{Z}[\zeta_{M}]$ lying above a fixed prime $p\equiv 1 \pmod {M}$, we have
	$$\mathbb{P}\left(	\mathrm{HD}_{\mathbb{K}_{4}\left(r\right)};1;\mathfrak{p}\right)+\iota_{\mathfrak{p}}(1/2)(-1)p =-\iota_{\mathfrak{p}}(r)(-64)  \cdot a_{p}\left(f_{	\mathrm{HD}_{\mathbb{K}_{4}\left(r\right)}}^{\sharp}\right),
	$$
	where $M:=\mathrm{lcd}(1/2,r)$. The corresponding values of $f_{	\mathrm{HD}_{\mathbb{K}_{4}\left(r\right)}}^{\sharp}$  are listed in Table~\ref{main-thm-table1}.
    \end{proposition}
    \begin{proof}
	We assume that $\iota_{\mathfrak{p}}\left(\frac{1}{p-1}\right)=\overline{\omega_{p}}$.
		Let
	\begin{align*}
		 U:=  \mathbb{P}\left(	\mathrm{HD}_{\mathbb{K}_{4}\left(r\right)};1;\overline{\omega_{p}}\right) +{\omega}_{p}^{-\frac{p-1}{2}}(-1)p +\overline{\omega_{p}}^{(p-1)r}(-64)  \cdot a_{p}\left(f_{	\mathrm{HD}_{\mathbb{K}_{4}\left(r\right)}}^{\sharp}\right).
	\end{align*}
	Since the value of $U$ lies in $\mathbb{Z}[\zeta_M]$, we write it as $\sum_{i=1}^m a_ib_i$, where $b_i\in\mathbb{Z}$ and  $\{a_1,a_2,\ldots ,a_m\}$ is an integral basis of $\mathbb{Z}[\zeta_M]$. It is noted that, $\iota_{\mathfrak{p}}(r)(-64) $ induces a quadratic character of $G(M)$ under the given assumption. Moreover, the values of  $a_p\left(f_{\mathrm{HD}_{\mathbb{K}_4\left(r\right)}}^{\sharp}\right)$ are integers with absolute values less than  $2p^{3/2}$, by Deligne bound. 
	\par We now proceed along the same lines as in Proposition~ 6.1 of \cite{allen_EHMM1}.
	By the explicit construction of the $\mathbb{K}_4$ functions from the hypergeometric background, together with the commutative formal group law (CFGL) and the supercongruence such as, \cite[Theorem~ 2.3]{allen_EHMM1}, combined with \eqref{H_q-definition}, we obtain
	$$U\equiv 0\pmod{p^2}.$$
    Since the family $\mathbb{K}_4(r)$ is \emph{Galois} for each $r\in\{1/8,1/12,1/24\}$, it follows that $\sigma_k(U)\equiv 0 \pmod {p^2}$ for every embedding $\sigma_k:\zeta_M\to\zeta_M^k$, where $k\in\left(\mathbb{Z}/M\mathbb{Z}\right)^\times$. Consequently, $p^2$ divides  $b_i$ for all $i=1, 2, \ldots, m$. Assume that at least one of the $b_i$'s is nonzero. Then the  norm of $U$ is at least $p^2$. On the other hand, by Theorem~\ref{thm:Katz}, any norm of $\mathbb{P}\left(	\mathrm{HD}_{\mathbb{K}_{4}\left(r\right)};1;\overline{\omega_{p}}\right)$ is less than or equal to $3p^{3/2}$. It follows that the norm of $U$ is at most $5p^{3/2}+p$. Hence, for $p\geq 29$, we must have $U=0$. The remaining finitely many primes can be verified by direct computation.
    \end{proof}
    We now prove Theorem~\ref{main-thm-K4} using  Theorem~\ref{thm-EHMM}.
    \begin{proof}[Proof of Theorem~\ref{main-thm-K4}]
    We  observe that the hypergeometric data  $\mathrm{HD}_{\mathbb{K}_{4}(r)}$, defined in \eqref{k4data}, satisfies the initial hypotheses of	Theorem~\ref{thm-EHMM} for all $$r \in \{1/2, 1/3, 1/4, 1/6, 1/8, 1/12, 1/24\}.$$ The cases $r=1/8, 1/12,$ and $1/24$ are established in Proposition~\ref{prop-main-thm}. Thus, it remains to analyze the cases corresponding to families $(1)-(4)$. Condition $(1)$ of Theorem~\ref{thm-EHMM} follows from the construction of the functions $\mathbb{K}_{4}(r)(N_{\mathbb{K}_{4}}(r)\tau)$ given in \eqref{K4eval}. Moreover, the families $(1)-(4)$ in Table~\ref{K4table} have at most two conjugates. 
   \par
    We now consider the condition $(2)$ of Theorem~\ref{thm-EHMM}. A direct computation shows that  $\iota_{\mathfrak{p}}(r)(-64)=1$ for $r=1/3,1/4$, and $\iota_{\mathfrak{p}}(r)(-64)=\iota_{\mathfrak{p}}(r)(-1)$ for $r=1/2, 1/6$. By Proposition~\ref{extendableprop} and Lemma~\ref{trans_P_function}, the functions $\mathbb{P}\left(	\mathrm{HD}_{\mathbb{K}_{4}\left(r\right)};1;\mathfrak{p}\right)$ are integer-valued for $r = 1/2,1/3,1/4,$ and $1/6$. Consequently, the hypotheses of Theorem~\ref{thm-EHMM} are satisfied in each of these cases. Therefore, applying Theorem~\ref{thm-EHMM}, we obtain
    $$\mathbb{P}\left(	\mathrm{HD}_{\mathbb{K}_{4}\left(r\right)};1;\mathfrak{p}\right)+\iota_{\mathfrak{p}}(1/2)(-1)p =-\psi(r)(\mathfrak{p})  \cdot a_{p}\left(f_{	\mathrm{HD}_{\mathbb{K}_{4}}\left(r\right)}^{\sharp}\right),$$
    where $\psi{(r)}(\mathfrak{p}) := \iota_{\mathfrak{p}}(r)(-64)$, and the explicit values of $f_{	\mathrm{HD}_{\mathbb{K}_{4}\left(r\right)}}^{\sharp}$ are given in the statement of Theorem~\ref{main-thm-K4}. 
   \end{proof}
   \begin{proof}[Proof of Theorem~\ref{main-thm-K5}]
	We assume that $\iota_{\mathfrak{p}}\left(\frac{1}{p-1}\right)=\overline{\omega_{p}}$. We first prove that \eqref{main-thm-K5-eq1} holds modulo $p$. To this end, we express $\mathbb{K}_5(r)$ as a power series in two distinct ways and use the fact that the associated commutative formal group laws are isomorphic; consequently, the $p$-th coefficients of the two expansions agree modulo $p$. For the first expansion, we write $t^r(1-t)^{-r-1/3}$ as a power series in $t^r$ using the binomial  expansion
	\begin{align*}
		(1-t)^{-r-1/3}=\sum_{k=0}^{\infty}\frac{\left(r+\frac{1}{3}\right)_k}{k!}t^k.
	\end{align*}
	Following \cite{allen_EHMM1}, set
	\begin{align*} A_r(k):=\frac{\left(r+\frac{1}{3}\right)_k}{k!}=\frac{\Gamma\left(r+\frac{1}{3}+k\right)}{\Gamma\left(r+\frac{1}{3}\right)\Gamma(k+1)}.
	\end{align*} 
Write $r=\frac{1}{e}$  with $e\geq 1$, and substitute $u=t^{\frac{1}{e}}$ in the differential form in \eqref{2F1integralK5}. This yields 
	\begin{align*}
	\gamma(t):&=27^{-r}t^{r-1}(1-t)^{-r-1/3}dt\\&=27^{-r}e \sum_{k=0}^{\infty}\frac{\left(r+\frac{1}{3}\right)_k}{k!}u^{ke+1} \frac{du}{u}=27^{-r}e\sum_{l=1}^{\infty} c_lu^{l-1}du,
	\end{align*} 
	where $c_l\in\mathbb{Q}$ and, in particular, $c_1=1$.  Using \cite[(5.4), (5.2)]{allen_EHMM2} together with the Gross-Koblitz formula, we obtain
	\begin{align*}
		A_r((p-1)r)\equiv -J_{\overline{\omega_{p}}}(r,2/3-r) \pmod p.
	\end{align*}
	Next, we express $\gamma(t)$ as a Fourier $q$-series. We do this by substituting $t=27q+O(q^2)$. We  write 
	\begin{align*}
		\gamma(t)=e\sum_{l=1}^{\infty} b_lq_e^{l-1}dq_e, 
	\end{align*}
	where $b_l\in\mathbb{Z}$ and $q_e:=q^{1/e}$. It is noted that the $q$-expansion(after re-scaling) is a Hecke eigenform for each Hecke operator $T_p$ when $p\equiv 1\pmod {M^\prime}$. Therefore, taking $R=\mathbb{Z}_p[27^{-r}]$ and $\sigma(27^{-r})=27^{-pr}$ in Proposition~\ref{CFGL-isom}, we have
	\begin{align}\label{main-thm-K5-eq2}
		-\iota_{\mathfrak{p}}(r)(1/27)a_{p}\left(f_{\mathrm{HD}_{\mathbb{K}_{5}\left(r\right)}}^{\sharp}\right)\equiv J_{\overline{\omega_{p}}}(r,2/3-r)\pmod p.
	\end{align}
	Since $r-2/3<0$, we rewrite it as $r+1/3$ when applying the Gross-Koblitz formula, which gives
	\begin{align*}
		-J_{\overline{\omega_{p}}}(1-r,-2/3+r)\equiv \pi^{(p-1)}\frac{\Gamma_p(1-r)\Gamma_p(r+1/3)}{\Gamma_p(1/3)} \pmod p. 
	\end{align*}
	Since $\pi^{p-1}=-p$, we have 
	\begin{align}\label{main-thm-K5-eq3}
		-\iota_{\mathfrak{p}}(1-r)(27)J_{\overline{\omega_{p}}}(1-r,r-2/3)\equiv 0 \pmod p.
	\end{align}
	Combining \eqref{main-thm-K5-eq2} and \eqref{main-thm-K5-eq3}, and applying \cite[(6.11)]{FL}, we deduce 
	\begin{align*}
	U_r:=a_{p}\left(f_{	\mathrm{HD}_{\mathbb{K}_{5}\left(r\right)}}^{\sharp}\right)+\iota_{\mathfrak{p}}(r)(27)\mathbb{P}\left(\mathrm{HD}_{\mathbb{K}_{5}\left(r\right)};1;\mathfrak{p}\right)+\iota_{\mathfrak{p}}(1-r)(27)\mathbb{P}\left(\overline{\mathrm{HD}_{\mathbb{K}_{5}\left(r\right)}};1;\mathfrak{p}\right)\\\equiv 0\pmod p.
   \end{align*}	 
  Using the properties of Jacobi sums, it is easy to verify that 
   \begin{align}\label{Jacobi_eq}
   	J_{\overline{\omega_{p}}}(1/4,5/12)=J_{\overline{\omega_{p}}}(1/4,1/12),~ J_{\overline{\omega_{p}}}(3/4,7/12)=J_{\overline{\omega_{p}}}(3/4,11/12).
   \end{align}
	  Since $U_r\in \mathbb{Z}[\zeta_{M^\prime}]$, we may write $U_r=\sum_{i=1}^s  u_iv_i$, where $u_i\in\mathbb{Z}$ and $\{v_1,v_2,\ldots,v_s\}$ is an integral basis of $\mathbb{Z}[\zeta_{M^\prime}]$. There are $\phi(M^\prime)$ different ways to embed $\mathbb{Z}[\zeta_{M^\prime}]$ to $\mathbb{Q}_p$. Combining \eqref{Jacobi_eq} (for the case $r=1/4$) with the integrality of $a_{p}\left(f_{\mathrm{HD}_{\mathbb{K}_{5}\left(r\right)}}^{\sharp}\right)$  for $p\equiv 1~\pmod {M^\prime}$, we obtain
	$\sigma_k\left(U_r\right) \equiv 0 \pmod p,$
	for each embedding $\sigma_k:\zeta_{M^\prime}\to \zeta_{M^\prime}^k$, where $k\in\left(\mathbb{Z}/M^\prime\mathbb{Z}\right)^\times$.
	 Therefore, $p$ divides each $u_i$. If any $u_i\neq 0$, then the norm of $U_r$ is at least $p$. On the other hand, by the bounds for Jacobi sums together with Deligne’s bound for newforms, the norm of $U_r$ is at most $4\sqrt{p}$. Thus, for $p\geq 17$, we must have $u_i=0$ for all $i$, and hence $U_r=0$. This proves \eqref{main-thm-K5-eq1} for $p\geq 17$, and the remaining primes can be verified by direct computation.
    \end{proof}
\section{Proof of Theorem~\ref{rel-appell-modu-1} and Theorem~\ref{rel-appell-modu-2}}
In this section, we prove Theorem~\ref{rel-appell-modu-1} and Theorem~\ref{rel-appell-modu-2}. The proofs proceed as applications of Theorem~\ref{main-thm-K5}, through which we establish explicit relations between the Fourier coefficients of weight two Hecke eigenforms and special values of the Appell series $F_1^p$ and $F_2^p$.
    \begin{proof}[Proof of Theorem~\ref{rel-appell-modu-1}]
	From \cite[Corollary~ 1.1]{Li_Li},  we have
	\begin{align*}
	F_1^p(a;b,b^\prime;c;x,x;\mathfrak{p})=p\cdot {_{2}}{F}_{1}\left(\begin{array}{cccc}
		\iota_{\mathfrak{p}}(b+b^\prime) & \iota_{\mathfrak{p}}(a)  \\
		&\iota_{\mathfrak{p}}(c)
	\end{array}|x
	\right),
   \end{align*}
	for any $a,b,b^\prime,c\in\mathbb{Q}$ and $x\in\mathbb{F}_p$, where  $\mathfrak{p}$ is a prime ideal in $\mathbb{Z}[\zeta_{M^\prime}]$ lying above a prime $p\equiv 1\pmod{\mathrm{lcd}{(a,b, b^\prime,c)}}$. 
	Putting $c=1$, $x=1$, and using the identity  \eqref{rel-gre-fu}, we deduce
	\begin{align*}
	F_1^p(a;b,b^\prime;1;1,1;\mathfrak{p})=\iota_{\mathfrak{p}}(a)(-1){\mathbb{P}}\left(\{b+b^\prime,a\};\{1,1\};1;\mathfrak{p}\right).
	\end{align*}
	Substituting $a=r, b=1/6, b^\prime=1/6 $ and $a=1-r, b=1/3, b^\prime=1/3 $, and adding the resulting identities, the desired result follows from Theorem~\ref{main-thm-K5}.
   \end{proof}
   \begin{proof}[Proof of Theorem~\ref{rel-appell-modu-2}]
   From the result of	He \emph{et al.} \cite[Theorem~ 1.4]{He-et-al}, we have
   \begin{align*}\label{main_thm_eq_1}
	&F_2^p(a;b,b^\prime;c,c^\prime;x,1;\mathfrak{p})
	=p^2\iota_{\mathfrak{p}}(b^\prime+c^\prime)(-1) {_{3}{F}_{2}}\left(\begin{array}{cccc}
		\iota_{\mathfrak{p}}(a) & \iota_{\mathfrak{p}}(b) & \iota_{\mathfrak{p}}(a-c^\prime) \\
		&\iota_{\mathfrak{p}}(c) & \iota_{\mathfrak{p}}(a+b^\prime-c^\prime)
	\end{array}|x
	\right),
  \end{align*}
   for any $a,b,b^\prime,c,c^\prime\in\mathbb{Q}$ and $x\in\mathbb{F}_p$, where  $\mathfrak{p}$ is a prime ideal in $\mathbb{Z}[\zeta_{M^\prime}]$ lying above a prime $p\equiv 1\pmod{\mathrm{lcd}{(a,b, b^\prime,c,c^\prime)}}$.
   Substituting $b^\prime=c=1$, $x=1$, and using the identity from \cite[Theorem~ 3.15 (ii)]{greene}, we obtain
  \begin{align*}
		&F_2^p(a;b,1;1,c^\prime;1,1;\mathfrak{p})\\
		&=-p\iota_{\mathfrak{p}}(c^\prime)(-1)_{2}{F}_{1}\left(\begin{array}{cccc}
		 \iota_{\mathfrak{p}}(a)& \iota_{\mathfrak{p}}(b)  \\
		&\iota_{\mathfrak{p}}(1) 
	\end{array}|1
	\right)+p^2\iota_{\mathfrak{p}}(c^\prime)(-1){\iota_{\mathfrak{p}}(c^\prime)\choose\iota_{\mathfrak{p}}(c^\prime-a)}{\iota_{\mathfrak{p}}(b-a+c^\prime)\choose\iota_{\mathfrak{p}}(c^\prime-a)}.
   \end{align*}
   The identity  \eqref{rel-gre-fu} yields
  \begin{align*}
 	&F_2^p(a;b,1;1,c^\prime;1,1;\mathfrak{p})\\
 	&=-\iota_{\mathfrak{p}}(b+c^\prime)(-1){\mathbb{P}}\left(\{a,b\};\{1,1\};1;\mathfrak{p}\right)+p^2\iota_{\mathfrak{p}}(c^\prime)(-1){\iota_{\mathfrak{p}}(c^\prime)\choose\iota_{\mathfrak{p}}(c^\prime-a)}{\iota_{\mathfrak{p}}(b-a+c^\prime)\choose\iota_{\mathfrak{p}}(c^\prime-a)}.
   \end{align*} 
  Taking $a=1/3, b=r$ and $a=2/3, b=1-r$, adding the resulting identities, the desired result follows from Theorem~\ref{main-thm-K5}.
   \end{proof}
\section{Acknowledgement}
We are deeply grateful to Ling Long, Fang-Ting Tu and Brian Grove for their many helpful discussions while working on this project. Their feedback and comments have significantly improved the quality of this article. We would like to thank the anonymous referee for the careful reading of the manuscript and providing several
helpful suggestions for improvements.

\end{document}